\documentclass[11pt,a4paper,twoside]{article}

\usepackage{amsfonts,amsthm,graphicx,verbatim}
\usepackage[margin=1.2in]{geometry}
\usepackage[colorlinks=true,citecolor=black,linkcolor=black,urlcolor=blue]{hyperref}

\graphicspath{{figs/}}

\usepackage{microtype}
\DisableLigatures[f]{encoding=*,family=*}

\renewenvironment{proof}{\noindent\textbf{Proof.}}{\hfill$\square$\medskip}
\newenvironment{remark}{\noindent\textbf{Remark.}}{\hfill$\diamondsuit$\medskip}

\newtheorem{lemma}{Lemma}
\newtheorem{thm}[lemma]{Theorem}
\newtheorem{cor}[lemma]{Corollary}

\newtheorem{obs}[lemma]{Observation}

\newcommand{\define}{\mathrel{:=}}
\newcommand{\curve}{\mathop\gamma}
\newcommand{\OO}{\mathop\mathcal{O}}
\newcommand{\pow}{\mathop\mathcal{P}}
\newcommand{\p}{\mathop{p}\nolimits}
\newcommand{\q}{\mathop{q}\nolimits}
\newcommand{\OTheta}{\mathop\Theta}
\newcommand{\Int}{\mathop\mathrm{Int}}
\newcommand{\f}{\mathop{f}}
\newcommand{\clique}{\mathop\omega}

\begin{document}

\title{Refining the Hierarchies of Classes of Geometric Intersection Graphs\thanks{A preliminary version of this work was presented at the Discrete Mathematics Days, Barcelona, July 6-8, 2016.}}
\author{Sergio Cabello\thanks{Department of Mathematics, IMFM, and Department of Mathematics, FMF, University of Ljubljana, Slovenia. Supported by the Slovenian Research Agency, programme P1-0297 and project L7-5459. E-mail: \texttt{sergio.cabello@fmf.uni-lj.si}.}
\and Miha Jej\v{c}i\v{c}\thanks{Faculty of Mathematics and Physics, University of Ljubljana, Slovenia. E-mail: \texttt{jejcicm@gmail.com}.}}
\maketitle

\begin{abstract}
	We analyse properties of geometric intersection graphs to show strict containment between some natural classes of geometric intersection graphs. 
	In particular, we show the following properties:
	\begin{itemize}
		\item A graph $G$ is outerplanar if and only if the 1-subdivision of $G$ is outer-segment.
		\item For each integer $k\ge 1$, the class of intersection graphs of segments with $k$ different lengths is a strict subclass of the class of intersection graphs of segments with $k+1$ different lengths.
		\item For each integer $k\ge 1$, the class of intersection graphs of disks with $k$ different sizes is a strict subclass of the class of intersection graphs of disks with $k+1$ different sizes.	
        \item The class of outer-segment graphs is a strict subclass of the class of outer-string graphs.
	\end{itemize}
\end{abstract}

%%%%%%%%%%%%%%%%%%%%%%%%%%%%%%%%%%%%%%%%%%%%%%%%%%%%%%%%%%%%%%%%%%%%%%%%%%%%%%%%%%%%%%%%%%%%%%%%%%%%%%%%%%%%%%%%%%%%
\section{Introduction}
For a set~$U$ and a nonempty family~$\mathcal{F}$ of subsets of~$U$ we define the \emph{intersection graph} of~$\mathcal{F}$, denoted by~$\Int(\mathcal{F})$, as the graph with vertex set~$V\left(\Int(\mathcal{F})\right)\define\mathcal{F}$, and edge set~$E\left(\Int(\mathcal{F})\right)\define\{AB\mid A,B\in \mathcal{F},\,A\neq B,\,A\cap B\neq\emptyset\}$. See Figure~\ref{fig:example} for an example. A graph~$G$ is an \emph{intersection graph} on~$U$ if there exists a family~$\mathcal{F}$ of subsets of~$U$ such that $G$ is isomorphic to~$\Int(\mathcal{F})$, and we call such a family~an \emph{intersection model} or an \emph{intersection representation} of~$G$. When considering \emph{geometric intersection graphs} we take~$\mathcal{F}$ to be a family of geometric objects, usually with~$U=\mathbb{R}^n$.  

\begin{figure}
	\centering
	\includegraphics[scale=1.2]{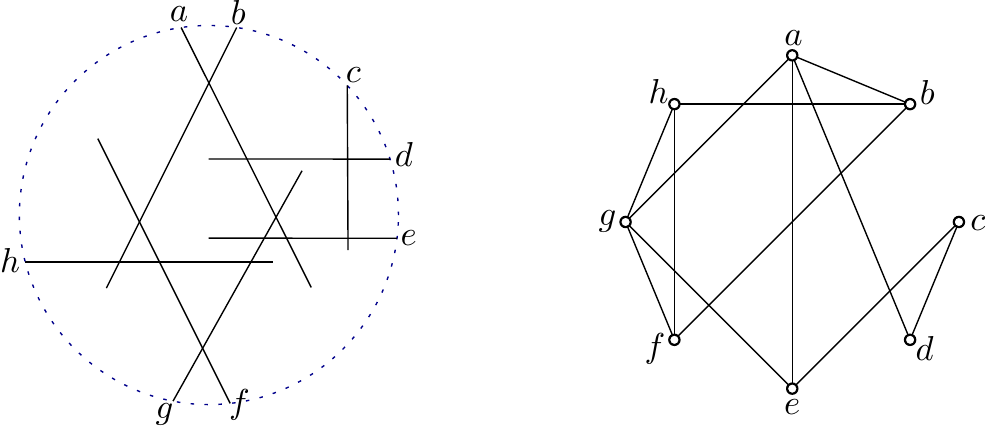}
	\caption{On the left there is a set of segments with one endpoint on a common circle. The intersection graph for these segments is shown on the right. This is an outer-segment graph.}
	\label{fig:example}
\end{figure}

For the rest of the paper we will restrict our attention to graphs defined by geometric planar objects, that is, $U=\mathbb{R}^2$. We consider the following types of (geometric intersection) graphs:
\begin{description}
	\item[string:] A graph is a \emph{string} graph if it is an intersection graph of curves in the plane.
	\item[outer-string:] A graph is an \emph{outer-string} graph if it has an intersection model consisting of curves lying in a disk such that each curve has one endpoint on the boundary of the disk.
	\item[circle:] A graph is a \emph{circle} graph if it is an intersection graph of chords in a circle.
	\item[segment:] A graph is a \emph{segment} graph if it is an intersection graph of straight-line segments in the plane.
	\item[ray:] A graph is a \emph{ray intersection} graph if it is an intersection graph of rays, or equivalently halflines, in the plane.
	\item[outer-segment:] A graph is an \emph{outer-segment} graph if it has an intersection model consisting of straight-line segments lying in a disk such that each segment has one endpoint on the boundary of the disk.
	\item[$k$-length-segment:] A graph is a \emph{$k$-length-segment} graph if it is an intersection graph of straight-line segments of at most $k$ different lengths in the plane. A $1$-length-segment graph is also called a \emph{unit-segment} graph.
	\item[disk:] A graph is a \emph{disk} graph if is an intersection graph of disks in the plane.
	\item[$k$-size-disk:] A graph is a \emph{$k$-size-disk} graph if it is an intersection graph of disks of at most $k$ different sizes in the plane. A $1$-size-disk graph is often called a \emph{unit-disk} graph.
\end{description}

Some relations between different classes of geometric intersection graphs are obvious. For example, ray intersection graphs are outer-segment graphs, which in turn are segment graphs and outer-string graphs. Similarly, unit-disk graphs are obviously disk graphs. The main focus of this paper is to provide properties showing that some of these containments between classes of geometric intersection graphs are proper. Although in some cases one can show strict containment using other properties, like for example \emph{$\chi$-boundedness}, our approaches are simpler, provide additional information, and in several cases provide a more detailed characterisation. A precise discussion of the alternative approaches is given in the relevant section. 

Here is a summary of our main results. 
\begin{itemize}
	\item The 1-subdivision of a graph $G$ is outer-segment if and only if it is outer-string and if and only if $G$ is outerplanar. This implies that the class of outer-string graphs is strictly contained in the class of string graphs. This is discussed in Section~\ref{sec:string}.
	\item For each integer $k\ge 1$, the class of $k$-length-segment graphs is strictly contained in the class of $(k+1)$-length-segment graphs. This is discussed in Section~\ref{sec:segment}.
	\item For each integer $k\ge 1$, the class of $k$-size-disk graphs is strictly contained in the class of $(k+1)$-size-disk graphs. This is discussed in Section~\ref{sec:disk}.
	\item The class of outer-string graphs is not contained in the class of segment graphs. As a consequence, the class of outer-segment graphs is strictly contained in the class of outer-string graphs. This is discussed in Section~\ref{sec:outer}.
\end{itemize}

There is much work on geometric intersection graphs and it is beyond the aims of this paper to provide a careful overview. We just emphasize some foundational articles, some big steps in the research, and some papers with a similar style. See the book by McKee and McMorris~\cite{MM99} for a general reference to intersection graphs.
Several key results on string graphs have been obtained by Kratochv\'{\i}l and Matou\v{s}ek~\cite{K91a,K91b,kratochvilExp,M14}; see the notes by Matou\v{s}ek~\cite{M13} for an overview. Showing that the recognition problem is in $\mathcal{NP}$ was only achieved relatively recently~\cite{SSS03}.
For segment graphs, cornerstone results include~\cite{CG09,KM94,MM13}. In particular, Kratochv\'{\i}l and Matou\v{s}ek~\cite{KM94} make a systematic study of the relation between different classes of geometric intersection graphs; the type of style we follow in this paper. Some key results for unit disk graphs include~\cite{BK98,CCJ90}. Janson and Kratochv\'{\i}l~\cite{JK92} analyse the probability that a random graph is some type of a geometric intersection graph. In their recent work, Chaplick et al.~\cite{CFHW15} also study the strict inclusion between different types of geometric intersection graphs with similar definitions. 

\paragraph{Notation.}\par For each natural number~$n$ we use~$[n]$ to denote the set~$\{1,\,2,\,3,\,\ldots,\,n\}$.

%%%%%%%%%%%%%%%%%%%%%%%%%%%%%%%%%%%%%%%%%%%%%%%%%%%%%%%%%%%%%%%%%%%%%%%%%%%%%%%%%%%%%%%%%%%%%%%%%%%%%%%%%%
\section{String vs outer-string}
\label{sec:string}

In this section we discuss properties that show that the class of outer-string graphs is a strict subclass of string graphs. Before providing our new approach, we discuss alternative approaches.

\subsection{Alternative approaches.}
\label{sec:string:alternative}
Kratochv\'{\i}l et al.~\cite{KGK} show that there are string graphs that are not outer-string. In particular, they show that there is a graph on 9 vertices that is not outer-string, while any graph on 11 vertices is a string graph. They also provide a neat characterisation of outer-string graphs: a graph~$G$ is an outer-string graph if and only if adding an arbitrarily connected clique to~$G$ results in a string graph. This implies that, if we take a minimal non-string graph, with respect to vertex deletion, and remove any vertex, we get a string graph that is not an outer-string graph.

Janson and Uzzell~\cite{JU14} show that the class of outer-string graphs is just a small fraction of all the string graphs. Their approach is based on graph limits. 

Another approach is based on~\emph{$\chi$-boundedness}. Let $\chi(G)$ and $\clique(G)$ denote the chromatic and the clique numbers of $G$, respectively. A class of graphs is~\emph{$\chi$-bounded} if there is some function $f\colon\mathbb{N}\rightarrow\mathbb{N}$ such that, for each graph~$G$ in the class, it holds that $\chi(G)\le\f(\clique(G))$. It is known that outer-string graphs are $\chi$-bounded~\cite{chi-bound}, while string graphs are not $\chi$-bounded~\cite{pawlik}. Therefore, both classes cannot be the same. This type of argument was provided by Bartosz Walczak in the context of segment graphs; we will encounter it again in Section~\ref{sec:segment}.

\subsection{Our approach.}
If in a graph we replace an edge $e$ with a path of length $2$ we say we $1$-subdivided the edge $e$. The \emph{$1$-subdivision} of a graph $G$ is the graph obtained from $G$ by $1$-subdividing each edge of $G$ once. A graph is~\emph{outerplanar} if it has a crossing-free embedding in the plane such that all vertices are on the same face. We have the following characterisation.

\begin{lemma}
	\label{outerstringCharacterization2}
	Let $G$ be a graph and let $H$ be its 1-subdivision. The following statements are equivalent:
	\begin{enumerate}
		\item $H$ is a circle graph.
		\item $H$ is outer-segment.
		\item $H$ is outer-string.
		\item $G$ is outerplanar.
	\end{enumerate}
\end{lemma}

Jean Cardinal noted that $(3)\iff(4)$ using the same type of approach that is used to show that a graph is planar if and only if its 1-subdivision is a string graph~\cite{schaefer,sinden}.

\begin{proof} $(1)\Longrightarrow(2)$ and $(2)\Longrightarrow(3)$ are trivial.

$(3)\Longrightarrow(4)$: Consider an intersection model of $H$ showing that $H$ is outer-string. For each vertex $u$ of $H$, let $\curve(u)$ be the curve corresponding to $u$ in the model. We may assume that there are no self-intersecting curves in the model, as otherwise we can find another model with this property. We use $w_{uv}$ for the vertex used to 1-subdivide $uv\in E(G)$. 

Consider first the curves $\curve(w_{uv})$ corresponding to the vertices $w_{uv}$ used to $1$-subdivide $G$. We call the curve corresponding to the vertex $w_{uv}\in V(H)$ the \emph{$uv$-curve}. For each $uv\in E(G)$, the curve~$\curve(w_{uv})$ intersects only $\curve(u)$ and $\curve(v)$, and no other curves of the model. Follow each such $uv$-curve from the boundary of the disk containing all curves, and remove all parts except for a part between two intersections with $\curve(u)$ and $\curve(v)$, containing no other intersections. Call this new curve~$\tilde{\curve}(uv)$. We are now left with a collection of pairwise non-intersecting outer-strings corresponding to vertices of $G$, and simple curves~$\tilde{\curve}(uv)$ corresponding to edges of $G$ between two different vertex-curves. Each~$\tilde{\curve}(uv)$ intersects $\curve(u)$ and $\curve(v)$ exactly once, and intersects no other curve. See Figure~\ref{fig:theFatteningAll}~a).

\begin{figure}[p]
	\centering
	\includegraphics[width=.9\textwidth]{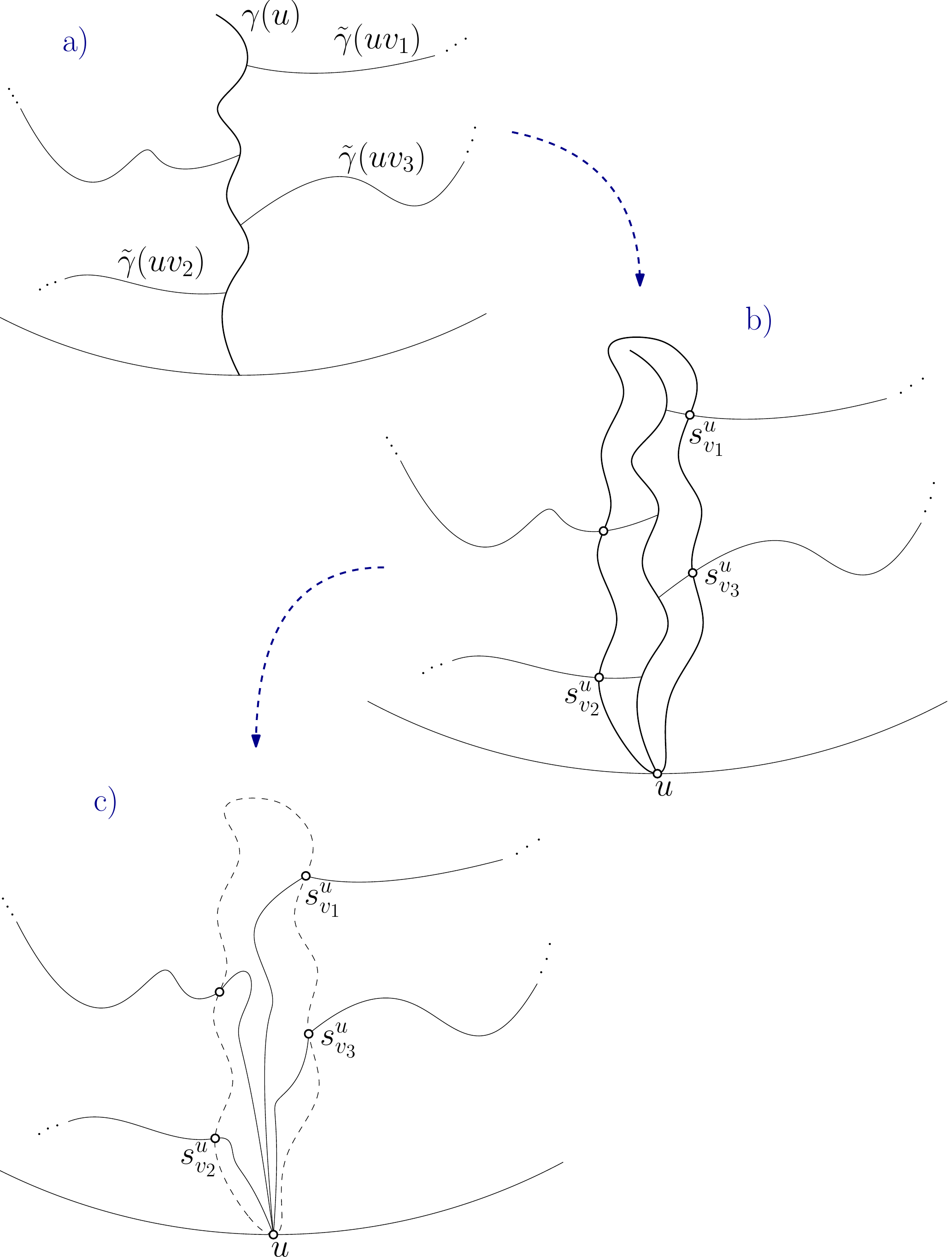}
	\caption{The implication $(3)\Longrightarrow(4)$ in Lemma~\ref{outerstringCharacterization2}.}
	\label{fig:theFatteningAll}
\end{figure}

Fatten each curve corresponding to $u\in V(G)$ slightly without creating any intersections between vertex-curves, and call such a region~$R_u$. For each vertex~$u$ of $G$ and each edge $uv\in E(G)$, let $s^u_v$ be the point where the curve $\tilde{\curve}(uv)$ first touches the region~$R_u$.

We are going to define now an embedding of $G$. Place each vertex $u$ of $G$ at the disk-endpoint of the corresponding curve $\curve(u)$. Within the region $R_u$ one can find $\deg_G(u)$ interior-disjoint paths connecting the vertex $u$ and the points $s^u_v$, where $uv\in E(G)$. See Figure~\ref{fig:theFatteningAll}~c). The paths between different regions $R_u$ do not intersect.
Each edge $uv$ of $G$ is drawn as the concatenation of the path within $R_u$ from $u$ to $s^u_v$, the curve $\tilde{\curve}(uv)$, and the path within $R_v$ from $s^v_u$ to $v$. We have obtained a crossing-free drawing of $G$ with all the vertices on the boundary of the disk and all the edges in the interior of the disk. Thus the drawing shows that $G$ is outerplanar.

$(4)\Longrightarrow(1)$: Consider an outerplanar graph $G$ with $n$ vertices, fix an outerplanar embedding of $G$, and let $v_0,\,v_1,\,\ldots,\,v_{n-1}$ be the vertices of $G$ along the facial walk of the outer face of the embedding, skipping repetitions of vertices. Placing the vertices $v_0,\,\ldots,\,v_{n-1}$ along a circle $C$, such that, for each vertex $v_i$, the vertices $v_{i-1}$ and $v_{i+1}$ are its predecessor and its successor along $C$ (where indices are modulo $n$), and drawing the edges of $G$ as straight-line segments, we get a planar drawing of $G$. In particular, we can place the vertices $v_0,\,\ldots,\,v_{n-1}$ at the corners of a regular $n$-gon inscribed in a circle~$C$. Consider one such outerplanar embedding of $G$ with straight-line segments. Note that $v_i$ and $v_{i+1}$ are not necessarily neighbours in $G$. 

Replace each vertex $v_i$ with a chord in $C$, with one endpoint $1/3$ of the way along $C$ towards $v_{i-1}$, and the other $1/3$ of the way along $C$ towards $v_{i+1}$. See Figure~\ref{fig:outerplanar2circleAll}.
For each vertex $v_i$ we also do the following. For each neighbour $v_j,\,j\neq i$ of $v_i$ in $G$, replace the straight-line drawing of the edge $v_iv_j$ with a chord in $C$ lying in the same position. In order to avoid these chords having a common endpoint, spread out their endpoints in the arc of $C$ centered on $v_i$ up to $1/3$ of the way along $C$ towards $v_{i-1}$ and up to $1/3$ of the way along $C$ towards $v_{i+1}$. See Figure~\ref{fig:outerplanar2circleAll}.

\begin{figure}[p]
	\centering
	\includegraphics{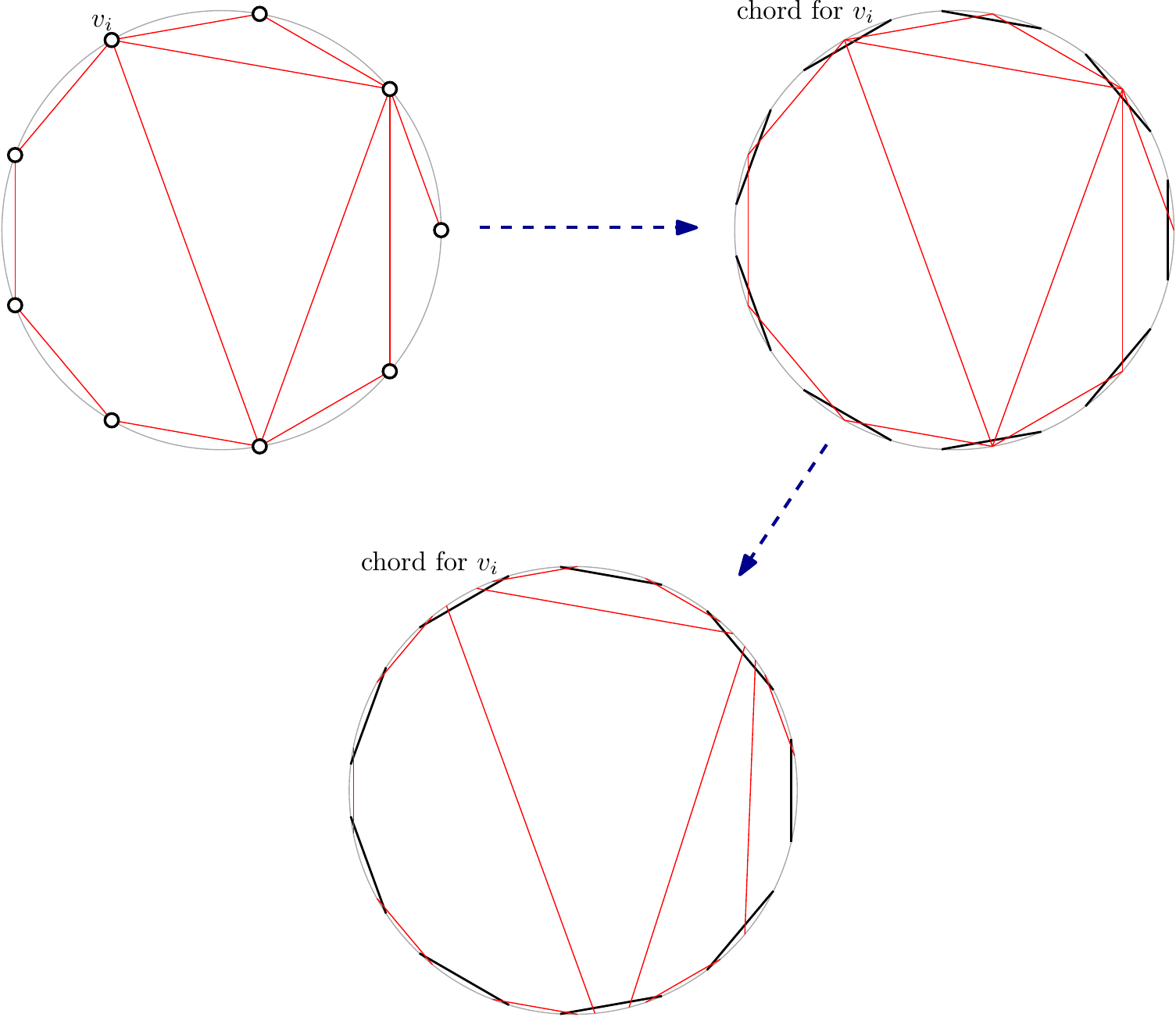}
	\caption{The implication $(4)\Longrightarrow(1)$ in Lemma~\ref{outerstringCharacterization2}.}
	\label{fig:outerplanar2circleAll}
\end{figure}

We claim that after performing this process for all the vertices $v_0,\,\ldots,\,v_{n-1}$, we have a model showing that $H$ is a circle graph. The chords corresponding to the vertices $v_0,\,\ldots,\,v_{n-1}$ are pairwise disjoint. A chord corresponding to the edge $v_iv_j$ intersects the chords for $v_i$ and $v_j$, and does not intersect any other chord. 
\end{proof}

\begin{remark}
The construction used in the implication $(4)\Longrightarrow(1)$ can be adapted to work in strongly polynomial time using Pythagorean triples, although the vertices do not form anymore the vertices of a regular $n$-gon. A \emph{Pythagorean triple} $(a,\,b,\,c)$ is a triple of positive integers such that $a\leq b\leq c$ and $a^2+b^2=c^2$. Pythagorean triples correspond to points on the unit circle with rational coordinates because, for a Pythagorean triple $(a,\,b,\,c)$ the point $(a/c,\,b/c)$ lies on the unit circle. It also holds that for each natural number $n$ there exists a \emph{hypotenuse} $c$, and at least $n$ different Pythagorean triples $(a_i,\,b_i,\,c)$ \cite{sierpinski}.

One method of generating $n$ triples with numbers in~$\OO\left(n^2\right)$ is by the parameterization~$(v^2-u^2,\,2uv,\,u^2+v^2)$, where $u$ and $v>u$ are relatively prime and of opposite parity~\cite{shanks}, which generates a set of distinct triples.

Both finding the desired hypotenuse and the triples for a given $n$ can be done in polynomial time. Therefore in polynomial time we can find $n$ distinct points on the unit circle with rational coordinates, and furthermore the coordinates are polynomially bounded. Those points are not the corners of a regular $n$-gon, but the construction can be done in strongly polynomial time. 
\end{remark}

It has been known~\cite{outerplanarAreCircle} that outerplanar graphs are circle graphs. Lemma~\ref{outerstringCharacterization2} implies that also the $1$-subdivision of an outerplanar graph is a circle graph. Note that the subdivision of an outerplanar graph is not outerplanar in general.

Using the characterisation in Lemma~\ref{outerstringCharacterization2}, and noting that ray graphs are outer-segment graphs, we have the following.

\begin{cor}
The 1-subdivision of $K_4$ is a string graph that is not outer-string. It is also a segment intersection graph that is not a ray intersection graph. 
\end{cor}

\begin{proof} $K_4$ is planar, therefore its 1-subdivision is a string graph; see~\cite{schaefer,sinden}. $K_4$ is not outerplanar, hence its 1-subdivision is not outer-string by Lemma~\ref{outerstringCharacterization2}. Since it is not outer-string, it is not a ray intersection graph. A segment representation is trivial to construct.
\end{proof}

%%%%%%%%%%%%%%%%%%%%%%%%%%%%%%%%%%%%%%%%%%%%%%%%%%%%%%%%%%%%%%%%%%%%%%%%%%%%%%%%%%%%%%%%%%%%%%%%%%%%%%%%%%%%%%%%%%%%%%%%%%%%%%%%%%%%%%%%%%%%%%%%%%%%%%%%%%%%%%%%%%%%%%%%%%%
%%%%%%%%%%%%%%%%%%%%%%%%%%%%%%%%%%%%%%%%%%%%%%%%%%%%%%%%%%%%%%%%%%%%%%%%%%%%%%%%%%%%%%%%%%%%%%%%%%%
\section{Segment graphs}
\label{sec:segment}

In this section we show that the class of $k$-length-segment graphs is a proper subclass of $(k+1)$-length-segment graphs. In particular, this implies that the class of unit-segment graphs is a proper subclass of the class of segment graphs.

In an informal discussion Bartosz Walczak suggested that $\chi$-boundedness, introduced in Section~\ref{sec:string:alternative}, implies some separation between subclasses of segment graphs. In particular, since unit-segment graphs are $\chi$-bounded~\cite{suk}, and segment graphs are not $\chi$-bounded~\cite{pawlik}, these two classes are different. The known rough upper estimates for $f(\omega(G))$ that are known in the context of $\chi$-boundedness are not strong enough to separate $k$ from~$k+1$ different lengths. In general, it is not clear how much such  approach can be pushed to distinguish $k$-length-segment graphs for different values of $k$.

Our approach is based on gadgets~\cite{KM94,orderingGadget} that can force a nested sequence of disjoint triangles (triples of segments intersecting pairwise).

For the notation in the following lemma, Figure~\ref{fig:ordGadgetSchematic} may be useful.

\begin{lemma}[\cite{orderingGadget}]
\label{lemma:orderingGadget}
Suppose we have Jordan curves $\ell$, $(\ell_i)_{i\in[n]}$, $\left(s^j_i\right)_{i\in[n-1],\,j\in[3]}$, and $(c_i)_{i\in[4n]}$ in the plane so that
\begin{enumerate}
\item $\ell$ crosses $\ell_i$, for every $i\in[n]$, and $s^2_i$, for every $i\in[n-1]$,
\item $c_i$ crosses $c_{i+1}$ ($c_1$ for $i=4n$) exactly once, for every $i\in[4n]$,
\item $\ell_i$ crosses $c_{2i}$ and $c_{4n-2i+2}$, for every $i\in[n]$,
\item both $s^1_i$ and $s^3_i$ cross $s^2_i$, for every $i\in[n-1]$,
\item $s^1_i$ crosses $c_{2i+1}$ and $s^3_i$ crosses $c_{4n-2i+1}$, for every $i\in[n-1]$,
\item the only other crossings among these curves are between pairs of $\ell_i$.
\end{enumerate}
Then the curves $\ell_i$ cross $\ell$ either in the order $\ell_1,\,\ldots,\,\ell_n$ or in the reverse of that order.%The conclusion remains true if instead of (1) we only require that (1a) $\ell$ crosses $\ell_i,\,i\in[n]$, and (1b) $s^2_i,\,i\in[n-1]$, \emph{may} cross $\ell$, but it does lie in the same connected component of $\mathbb{R}^2\setminus\bigcup_{i\in[4n]}c_i$ as $\ell$.
\end{lemma}

\begin{figure}
	\centering
	\includegraphics[width=.8\textwidth]{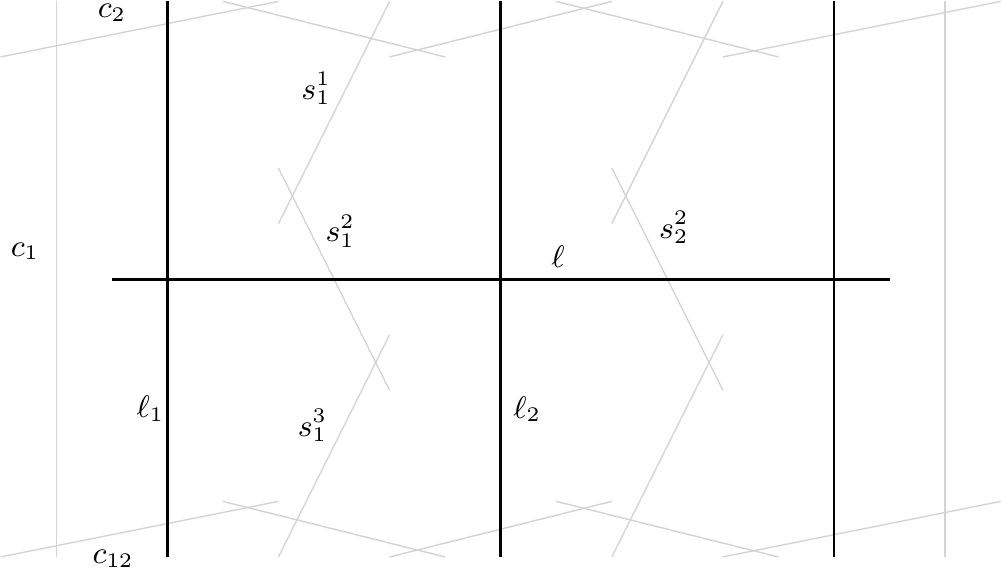}
	\caption{Construction in Lemma~\ref{lemma:orderingGadget} for $n=3$.}
	\label{fig:ordGadgetSchematic}
\end{figure}

We call the collection of curves from Lemma~\ref{lemma:orderingGadget} an \emph{ordering gadget}. We further define a graph~$H$, to which we will be referring in this section, as follows. First take the intersection graph of ordering gadget from Lemma~\ref{lemma:orderingGadget} with $n=3$. We add to the graph two new vertices, $a$ and $b$, such that (i) $a$ is adjacent to  $b,\, \ell,\,\ell_1,\,s^1_1,\,\ell_2,\,\mathrm{and}\ s^3_2$, and (ii) $b$ is adjacent to $a,\, \ell,\,\ell_1,\,s^3_1,\,\ell_2,\,\mathrm{and}\ s^1_2$. This finishes~the description of the graph~$H$. As shown in Figure~\ref{fig:ordGadgetSchematicAB}, the graph $H$ is a segment intersection graph. For simplicity, we denote the segments (or vertices) $\ell_1$ and $s^2_1$ by $c$ and $z$ respectively, and notice that the segments $a,\,b$, and $c$ define a triangle ($K_3$) in the graph. In any intersection model of $H$ with segments, there is a unique (geometric) triangle contained in the union of $a,\,b,\,c$.

\begin{figure}
	\centering
	\includegraphics[width=.8\textwidth]{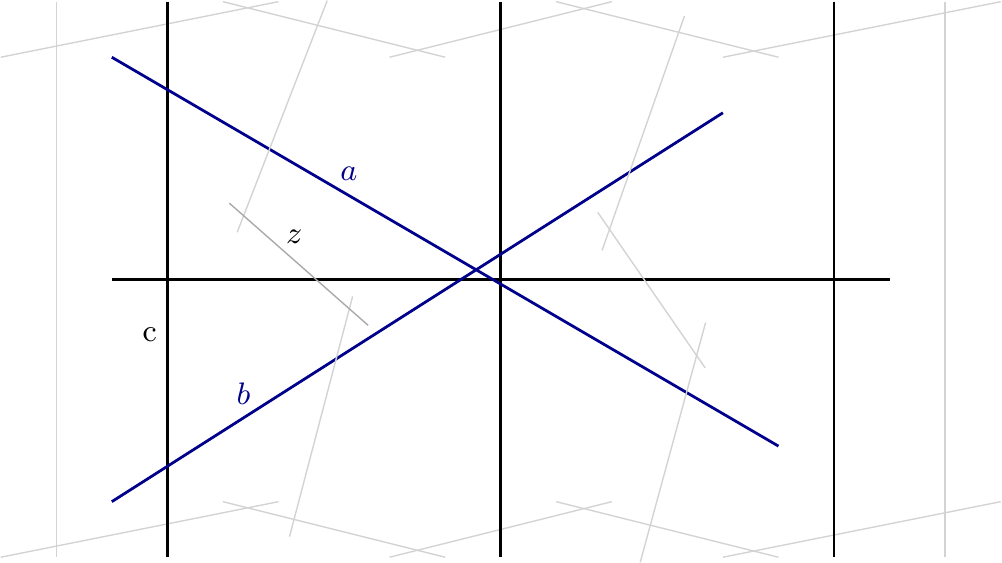}
	\caption{Intersection model with segments for the graph~$H$.}
	\label{fig:ordGadgetSchematicAB}
\end{figure}

\begin{lemma}
\label{lemma:triangle}
In any intersection model of~$H$ with segments, the segment~$z$ is contained in the triangle defined by $a,\,b,\,c$.
\end{lemma}

\begin{proof} The set $\bigcup_{i\in[12]}c_i$ contains a unique closed Jordan curve $\gamma$ that separates the plane in two faces. Since $\ell,\,z,\,\mathrm{and}\ s^2_2$ do not cross any of the curves $c_i$, they must all lie in the same face~$F$ defined by $\gamma$. Considering the face inside~$F$ defined by the unique closed Jordan curve contained in the union of $\left\{\ell_1,\ell_2\right\}\cup\bigcup_{i=2}^4 c_i\cup\bigcup_{i=10}^{12}c_i$, one can argue that the conditions and the conclusion in Lemma~\ref{lemma:orderingGadget} imply that the segment $\ell$ must cross the segments $c=\ell_1,\,z,\,\ell_2$ in that order or in the reverse order. Similarly $\ell$ must cross the segments 
$\ell_2,\,s^2_2,\,\ell_3$ in that order or in the reverse order. We conclude that~$\ell$ crosses the 
other segments in the order $c,\,z,\,\ell_2,\,s^2_2,\,\ell_3$ (or in the reverse order) in the ordering gadget. (In fact this is used in the proof~\cite{orderingGadget} of Lemma~\ref{lemma:orderingGadget}.)

Let $F'$ be the face contained in $F$ defined by the unique closed Jordan curve contained in $\bigcup_{i=1}^3\left\{s^i_1,\,s^i_2\right\}\cup\bigcup_{i=3}^5c_i\cup\bigcup_{i=9}^{11}c_i$. See Figure~\ref{fig:ordGadgetSchematicFaces}. All three pairwise crossings between $a$, $b$, and $\ell$ must occur in $F'$ because the segments crossed by $a$, $b$, and $\ell$ alternate on the boundary of $F'$. In particular, the triangle defined by $a$, $b$, and $c$ contains a portion of $F'$. This portion has $z=s^2_1$ on the boundary because $\ell$ crosses $c$ and $z$ consecutively, after (or before) crossing $a$ or $b$. Therefore, the triangle defined by $a$, $b$, and $c$ must contain $z$.
\end{proof}

\begin{figure}
	\centering
	\includegraphics[width=.8\textwidth]{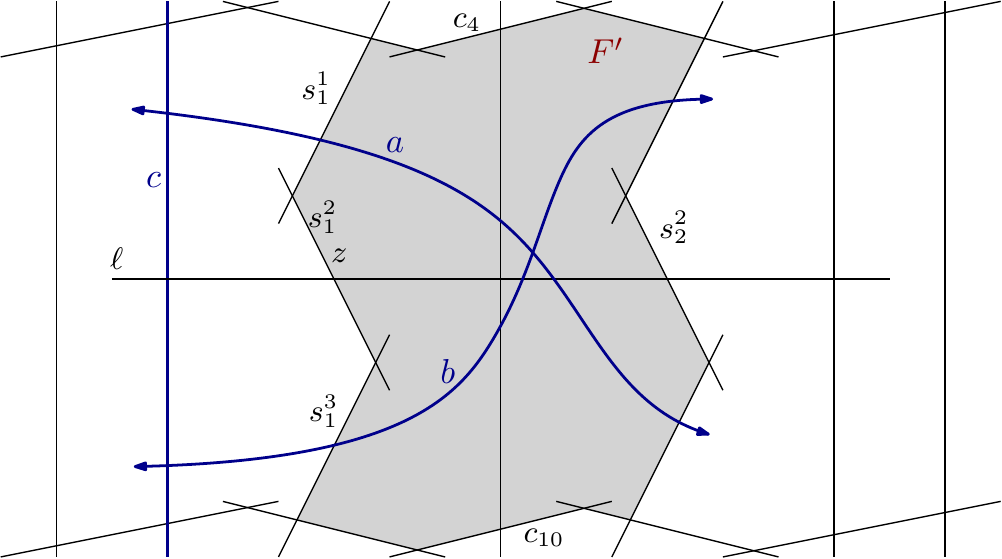}
	\caption{Face~$F'$ used in the proof of Lemma~\ref{lemma:triangle}.}
	\label{fig:ordGadgetSchematicFaces}
\end{figure}

On the one hand, $H$ is a segment graph, as shown in Figure~\ref{fig:ordGadgetSchematicAB}. On the other hand, Lemma~\ref{lemma:triangle} implies $H$ cannot be a unit-segment intersection graph because we cannot have a unit-length segment inside the geometric triangle defined by three unit-length segments. We will generalize this below to obtain a finer classification.

\begin{lemma}
\label{lemma:ordGadgetUnitSegs}
The graph~$H$ has an intersection model with segments such that: all segments but $z$ have unit length, the segment~$z$ is contained in the triangle defined by $a,\,b,\,c$, and the distance from any point on $z$ to any point on $a\cup b\cup c$ is at least $C$~times the length of~$z$, where $C$ is an arbitrary constant we can choose.
\end{lemma}

\begin{proof}
	See Figure~\ref{fig:ordGadgetUnit}.
\end{proof}

\begin{figure}
  \centering
  \includegraphics[width=0.5\textwidth]{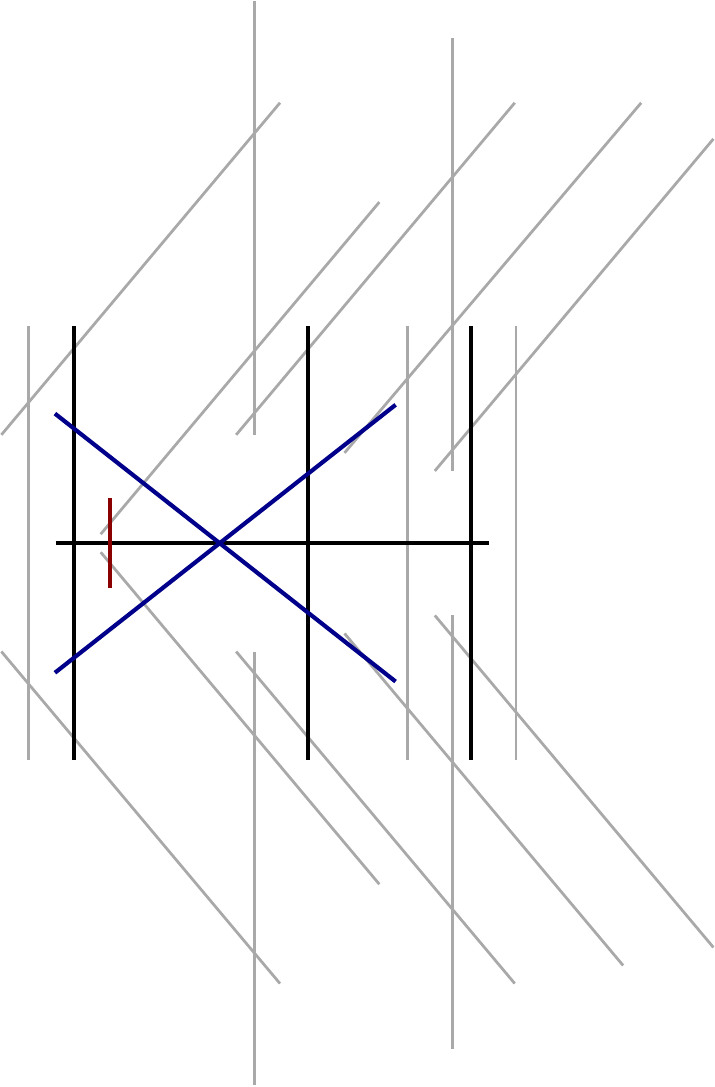}
  \caption{Intersection model of~$H$ where all segments but $z$ have unit length. The segment~$z$ can be made arbitrarily small by bringing $s^1_1$ and $s^3_1$ closer.}
   \label{fig:ordGadgetUnit}
\end{figure}

\begin{thm}
For any natural number~$k$, the class of $k$-length-segment graphs is a strict subclass of $(k+1)$-length-segment graphs.
\end{thm}

\begin{proof}
For each~$k$, we define a graph~$G_k$ such that $G_k$ can be realized as an intersection graph of segments of $k+1$ different lengths, and cannot be realized as an intersection graph of segments of $k$ different lengths.

We define $G_k$ as follows. We fix an intersection model~$\mathcal{F}$ for $H$ as described in Lemma~\ref{lemma:ordGadgetUnitSegs}, with the constant $C$ large enough. Since all the segments in $\mathcal{F}$, but $z$, have unit length and the graph $H$ has small diameter, all the segments of $\mathcal{F}$ are contained in a disk of radius 20.

We make disjoint copies $\mathcal{F}_1,\,\ldots,\,\mathcal{F}_k$ of the family~$\mathcal{F}$. For each $i\in [k]$, we denote by $a_i,\,b_i,\,c_i,\,z_i$ the copies of $a,\,b,\,c,\,z$ in $\mathcal{F}_i$, respectively.
Then, we define families of segments $\mathcal{F}'_1,\,\ldots,\,\mathcal{F}'_k$ inductively.
We take $\mathcal{F}'_1= \mathcal{F}_1$, $a'_1=a_1$, $b'_1=b_1$, $c'_1=c_1$ and $z'_1=z_1$. For each $i\in [k-1]$, we define $\mathcal{F}'_{i+1}$ by scaling and applying a rigid motion to $\mathcal{F}_{i+1}$ such that the segment $a_{i+1}$ becomes the segment $z'_{i}$ of $\mathcal{F}'_{i}$. We denote by $a'_{i+1},\,b'_{i+1},\,c'_{i+1},\,z'_{i+1}$ the transformed version of $a_{i+1},\,b_{i+1},\,c_{i+1},\,z_{i+1}$ in $\mathcal{F}'_{i+1}$. Note that some segments of $\mathcal{F}'_{i+1}$ intersect some segments of $\mathcal{F}'_{i}$. However, since all the segments of $\mathcal{F}_{i+1}$ are contained in a disk of radius 20 times the length of $a'_{i+1}=z'_{i}$, we can take the constant~$C$ in Lemma~\ref{lemma:ordGadgetUnitSegs} large enough so that no segment of $\mathcal{F}'_{i+1}$ can intersect $a'_i$, $b'_i$, or $c'_i$.
Finally, we take $G_k$ to be the intersection graph of $\bigcup_{i\in[k]}\mathcal{F}'_i$. (Whether we keep both $a_{i+1}$ and $z_i$ or only one of them is not very relevant. For the discussion, it is more convenient that we identify them.)

At the level of abstract graphs, the construction of $G_k$ can be seen as taking $k$ disjoint copies $H_1,\,\ldots,\,H_k$ of $H$, and identifying the copy $z_i$ of $z$ in $H_i$ with the copy $a_{i+1}$ of $a$ in $H_{i+1}$ (where $i=1,\,\ldots,\,k-1$). We also need some additional edges between the vertices of $H_{i}-\{a_i,\,b_i,\,c_i\}$ and $H_{i+1}$ to make an intersection model with segments possible, but they are not relevant in the discussion. It is only important that there is no edge between $\{a_i,\,b_i,\,c_i\}$ and $H_{i+1}$ for each $i\in [k-1]$.

The graph~$G_k$ is a $(k+1)$-length-segment graph by construction. Indeed, the lengths of segments in $\mathcal{F}_i$ decrease with $i$, and, for each $i\in[k]$, the family  $\bigcup_{j\in[i]}\mathcal{F}_j$ has segments of exactly $i+1$ different lengths.

It remains to show that the graph~$G_k$ is not a $k$-length-segment graph. Consider an intersection model of $G_k$ with segments. For each $i\in[k]$, let $\tilde\mathcal{F}_i$ be the restriction of the model to the graph~$H_i$, let $\lambda_i$ be the length of the longest segment in $\tilde\mathcal{F}_i$, and let $T_i$ be the geometric triangle defined by the segments that correspond to $a_i,\,b_i,\,c_i$.
For each $i\in [k]$, Lemma~\ref{lemma:triangle} implies that the segment for $z_{i}$ is contained in $T_i$. Since $z_i=a_{i+1}$, the graph~$H_{i+1}$ is connected, and there are no edges between $H_{i+1}$ and $\{a_i,\,b_i\,,c_i\}$, all the segments of $\tilde\mathcal{F}_{i+1}$ are contained in $T_i$. It follows that the longest segment in $\tilde\mathcal{F}_{i+1}$ has to be shorter than the longest edge of the triangle~$T_i$, which in turn is shorter than the longest segment in $\tilde\mathcal{F}_{i}$. Thus we have $\lambda_1>\lambda_2>\cdots>\lambda_k$. Moreover, Lemma~\ref{lemma:triangle} for $\tilde\mathcal{F}_k$ implies that $\tilde\mathcal{F}_k$ has segments of at least 2 different lengths. This is, $\tilde\mathcal{F}_k$ has a segment of length $\lambda_k$, and another segment of length $\lambda'_k<\lambda_k$. With this we have shown that the intersection model has at least the $k+1$ different lengths $\lambda_1>\cdots>\lambda_k>\lambda'_k$.
\end{proof}

%%%%%%%%%%%%%%%%%%%%%%%%%%%%%%%%%%%%%%%%%%%%%%%%%%%%%%%%%5%%%%%%%%%%%%%%%%%%%%%%%%%%
%%%%%%%%%%%%%%%%%%%%%%%%%%%%%%%%%%%%%%%%%%%%%%%%%%%%%%%%%5%%%%%%%%%%%%%%%%%%%%%%%%%%
\section{Disk graphs}
\label{sec:disk}

In this section we show that the class of $k$-size-disk graphs is a proper subclass of the $(k+1)$-size-disk graphs. In particular, this implies that the class of unit-disk graphs is a proper subclass of the class of disk graphs. An alternative way to show the separation between disks and unit disks is given by McDiarmid and M{\"{u}}ller~\cite{MM14}, where they provide near-tight bounds on the number of disk and unit-disk graphs with $n$ vertices. It follows from their result that there are (many) disk graphs that are not unit-disk graphs.

The argument here is much simpler and uses the following folklore result.

\begin{obs}
\label{obs:circlePacking}
The star $K_{1,\,6}$ is a disk graph, it has an intersection model with disks of two sizes, but it is not a unit-disk graph. In any intersection model of $K_{1,\,6}$ with disks, at least one of the non-central vertices must be represented by a disk of strictly smaller size than the central disk.
\end{obs}

\begin{proof} Standard plane geometry shows that if $6$ unit disks~$D_1,\,\ldots,\,D_6$ intersect a unit disk $D$, then some pair $D_i,\,D_j$ must also intersect. Indeed, the angle of the straight-line segments connecting the centre of~$D$ to the centres of~$D_1,\,\ldots,\,D_6$ must contain two segments that form an angle of at most $60$ degrees, and the two unit disks defining those segments must intersect. This is related to the \emph{kissing number} in the plane.
\end{proof}

\begin{thm}
For every natural number $k\ge 2$, the class of $(k-1)$-size-disk graphs is a proper subclass of $k$-size-disk graphs.
\end{thm}

\begin{proof}
For each $k\in\mathbb{N}$ we construct a disk graph~$G_k$ that requires the disks to be of at least $k$ different sizes (and can be realized by disks of \emph{exactly} $k$ different sizes). For $k=2$ we simply take $K_{1,\,6}$ by Observation~\ref{obs:circlePacking}.

For $k\geq3$, we take for $G_k$ six copies of $G_{k-1}$, each with its own distinguished central vertex, to which we add a single vertex $v_k$ connected to the six central vertices. It is easy to see that $G_k$ is a $k$-size disk graph. See Figure~\ref{fig:disks}. The graph $G_k$ is a rooted $6$-ary tree with $k$ levels.

\begin{figure}
	\centering
	\includegraphics[width=0.4\textwidth]{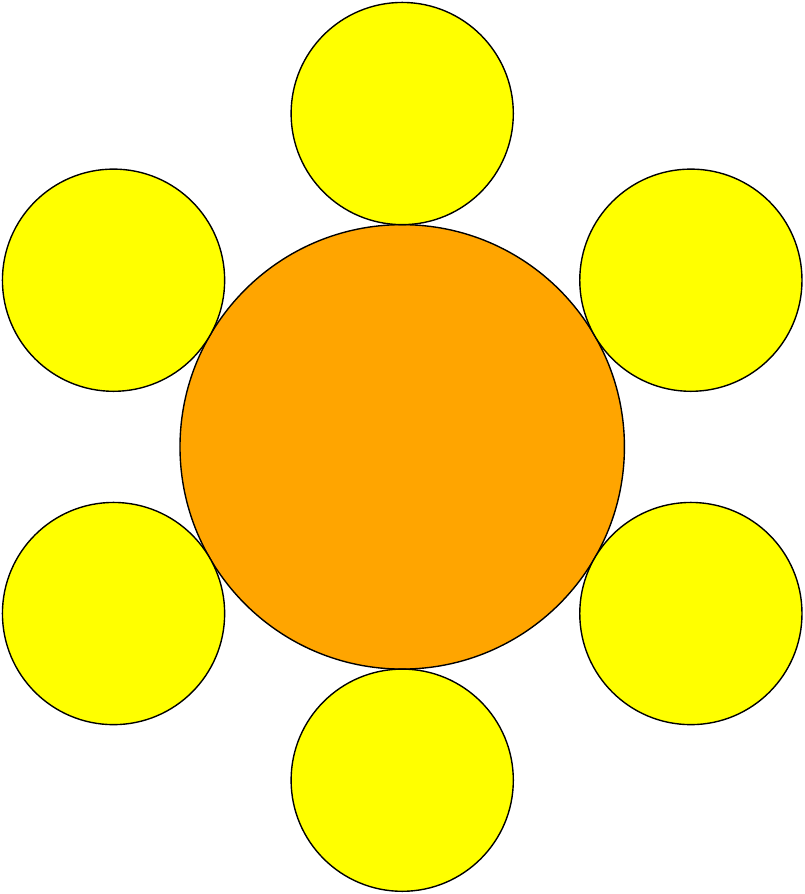}\hspace{2cm}
	\includegraphics[width=0.4\textwidth]{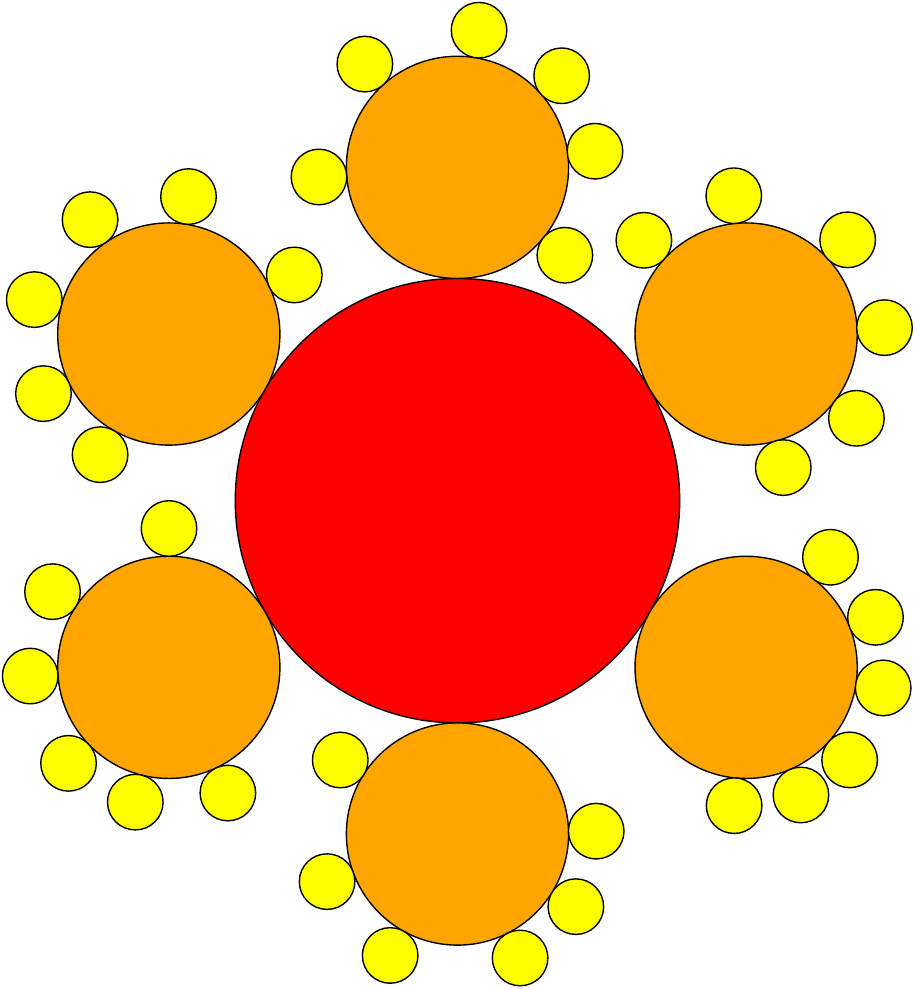}
	\caption{The left shows~$G_2$, the right shows~$G_3$.}
	\label{fig:disks}
\end{figure}

To see that drawing the disk representation of $G_k$ in the plane requires $k$ different sizes of disks, consider any intersection model with disks. Let $\tilde{D}_1$ be the disk representing the central vertex $v_k$ of $G_k$. The vertices $u_1,\,\ldots,\,u_6$ adjacent to $v_k$ in $G_k$ are represented by disks~$\{D_{2,\,j}\}_{j\in[6]}$ that must all intersect $\tilde{D}_1$ and must not intersect among themselves. Therefore, by Observation~\ref{obs:circlePacking}, at least one of them, let's call it $\tilde{D}_2$, is of a size strictly smaller than $\tilde{D}_1$. We repeat the same argument at an arbitrary ``level''~$i$; at least one of the disks~$\{D_{i,\,j}\}_{j\in[6]}$ must be of a strictly smaller size than $\tilde{D}_{i-1}$ in order to be able to intersect it, as well as to avoid intersecting among themselves. We denote it by $\tilde{D}_i$. This means that in the intersection model we have disks $\tilde{D}_1,\,\ldots,\,\tilde{D}_k$ with decreasing size. Thus, we have disks of $k$ different sizes.
\end{proof}

%%%%%%%%%%%%%%%%%%%%%%%%%%%%%%%%%%%%%%%%%%%%%%%%%%%%%%%%%5%%%%%%%%%%%%%%%%%%%%%%%%%%
%%%%%%%%%%%%%%%%%%%%%%%%%%%%%%%%%%%%%%%%%%%%%%%%%%%%%%%%%%%%%%%%%%
\section{Outer-string versus outer-segment}
\label{sec:outer}
In this section we show that the class of outer-segment graphs is a proper subclass of outer-string graphs.

Let us discuss first the obvious approach based on previous works. The common way to show that there are string graphs that are not segment graphs is to show that there are some string graphs such that, in any model, some pairs of curves have to intersect multiple times. Thus, such a graph cannot have a model using segments. This is related to the concept of weak realizations and was developed in~\cite{kratochvilExp}; see~\cite{M13} for a recent account. At first glance it seems that the approach does not match well with the concept of outer-representations.

Our approach here to distinguish classes of graphs is based on counting arguments to bound the number of different neighbourhoods of the vertices in each of the classes. Such an approach is implicitly based on the concept of VC-dimension, although our eventual presentation uses polynomials. It provides also a new (weaker) tool to separate segment graphs from string graphs.

We will use the expression that a set of segments in the plane is in \emph{general position} if and only if no three endpoints of said segments are collinear.

\begin{lemma}\label{lemma:segsCross}
Fix a segment~$\overline{AB}$ in the plane. Then there exist two polynomials~$\p_{AB}$ and $\q_{AB}$ of degree at most four such that, for each segment~$\overline{CD}$ in general position with respect to~$\overline{AB}$, segments~$\overline{AB}$ and $\overline{CD}$ cross if and only if $\p_{AB}(C,\,D)<0$ and $\q_{AB}(C,\,D)<0$.
\end{lemma}

\begin{proof}
In this proof we denote the coordinates of $A$ by~$(a_1,\,a_2)$, of $B$ by~$(b_1,\,b_2)$, and similarly for points $C$ and~$D$.

The triple $(A,B,C)$ of points is in counterclockwise order if and only if the \emph{signed area} of the triangle $ABC$ is negative, that is
$$
	\frac{1}{2}\left|\begin{array}{ccc}
	a_1&a_2&1\\
	b_1&b_2&1\\
	c_1&c_2&1
	\end{array}\right|<0.
$$
This condition can also be written as
$$
	(c_2-a_2)(b_1-a_1)>(b_2-a_2)(c_1-a_1).
$$
Note that the order of the points in the triple is important: if $(A,B,C)$ is in counterclockwise order, then $(A,C,B)$ is not in counterclockwise order.

Two segments in general position $\overline{AB}$ and $\overline{CD}$ cross if and only if the points $A,\,B,\,C$, and $D$ satisfy the following:
\begin{enumerate}
\item exactly one of the triples $(A,\,C,\,D)$ and $(B,\,C,\,D)$ is in counterclockwise order, and
\item exactly one of the triples $(A,\,B,\,C)$ and $(A,\,B,\,D)$ is in counterclockwise order.
\end{enumerate}

This gives us the following equivalence: Segments $\overline{AB}$ and $\overline{CD}$ in general position cross if and only if
\begin{enumerate}
\item $\Big((d_2-a_2)(c_1-a_1)-(c_2-a_2)(d_1-a_1)\Big)\Big((d_2-b_2)(c_1-b_1)-(c_2-b_2)(d_1-b_1)\Big)<0$, and
\item $\Big((c_2-a_2)(b_1-a_1)-(b_2-a_2)(c_1-a_1)\Big)\Big((d_2-a_2)(b_1-a_1)-(b_2-a_2)(d_1-a_1)\Big)<0$.
\end{enumerate}

We therefore have two polynomial inequalities of degree at most four in four variables to decide whether two segments cross.
\end{proof}

Let~$S$ be a set of segments and~$T\subseteq S$. A segment~$\gamma$, which may or may not belong to~$S$, is an \emph{exact transversal of~$T$ in~$S$} if and only if $\gamma$ intersects all the segments of~$T$, and $\gamma$ is disjoint from each segment of~$S\setminus T$.

\begin{lemma}\label{lemma:numberOfSeparableSubsetsOfSegs}
Let~$S$ be a set of $n$~segments in the plane. Consider the family
\begin{eqnarray*}
\mathbb{S}&=&\{T\subseteq S\mid\mathrm{there\ exists\ an\ exact\ transversal\ of}\ T\ \mathrm{in}\ S\}\\
&=&\{T\subseteq S\mid\exists\mathrm{\ a\ segment}\ \gamma\colon(\forall t\in T\colon\gamma\cap t\neq\emptyset\wedge\forall t'\in S\setminus T\colon\gamma\cap t'=\emptyset)\}.
\end{eqnarray*}
Then $|\mathbb{S}|\in\OO\left(n^4\right)$.
\end{lemma}

\begin{proof}
Let~$p_1,\,\ldots,\,p_m$ be polynomials on~$\mathbb{R}^d$, and let~$\sigma\in\{-1,\,0,\,1\}^m$ be a vector. If there exists an~$x\in\mathbb{R}^d$ such that for all~$i$ the sign of~$p_i$ in~$x$ is~$\sigma_i$, then~$\sigma$ is called a \emph{sign pattern of~$p_1,\,\ldots,\,p_m$}.

Following Matou\v{s}ek~\cite[Chapter 6]{matousek} and references therein, the maximum number of sign patterns for a collection $p_1,\,p_2,\,\ldots,\,p_m$ of $d$-variate polynomials of degree at most $D$ is bounded by
$$
\left(\frac{50Dm}{d}\right)^d.
$$

Let $\gamma_1,\,\ldots,\,\gamma_k$ be exact transversals of sets~$T_1,\,\ldots,\,T_k$ in $S$, respectively, such that $\left\{T_i\mid i\in[k]\right\}=\mathbb{S}$. We now modify $S\cup\{\gamma_1,\,\ldots,\,\gamma_k\}$ to general position by enlarging each segment by a small amount without introducing any new intersections. Now we have segments in general position; we use~$S'$ to denote the new, modified set.

Consider a fixed segment~$\overline{AB}\in S'$. By Lemma~\ref{lemma:segsCross} there exist polynomials~$\p_{AB}$ and $\q_{AB}$ of degree at most four such that a segment~$\overline{CD}$ crosses $\overline{AB}$ if and only if $\p_{AB}(C,\,D)<0$ and $\q_{AB}(C,\,D)<0$.

We therefore consider the set $P=\bigcup_{\overline{AB}\in S'}\{\p_{AB},\,\q_{AB}\}$ of $2n$~polynomials in four~variables of degree at most four. The set~$P$ defines a number of sign patterns, which is an upper bound for $|\mathbb{S}|$.

In our case this gives us the result $|\mathbb{S}|\leq(100n)^4\in\OTheta\left(n^4\right)$.
\end{proof}

\begin{thm}\label{thm:outerstringNotSegment}
There are outer-string graphs that are not segment graphs.
\end{thm}

\begin{proof}
For a natural number~$k$ we construct a graph~$G_k$ in the following way. Take $V(G_k)\define[k]\cup\pow[k]$, where $\pow[k]$ is the power set of~$[k]$. We denote vertices in~$[k]$ by $1,\,2,\,3,\,\ldots$, and vertices in~$\pow[k]$ by $A,\,B,\,C,\,\ldots$ Take
$$
E(G_k)\define\{AB\mid A,\,B\in\pow[k]\}\cup\{iA\mid i\in[k],\,A\in\pow[k],\,i\in A\}.
$$
Thus $G_k$ is a graph on $2^k+k$~vertices. It is easy to see that $G_k$ is an outerstring graph (see Figure~\ref{fig:outerstringGraph}). We intend to prove that $G_k$ is not a segment intersection graph for a sufficiently large number~$k$.

\begin{figure}
	\centering
	\includegraphics[scale=.9]{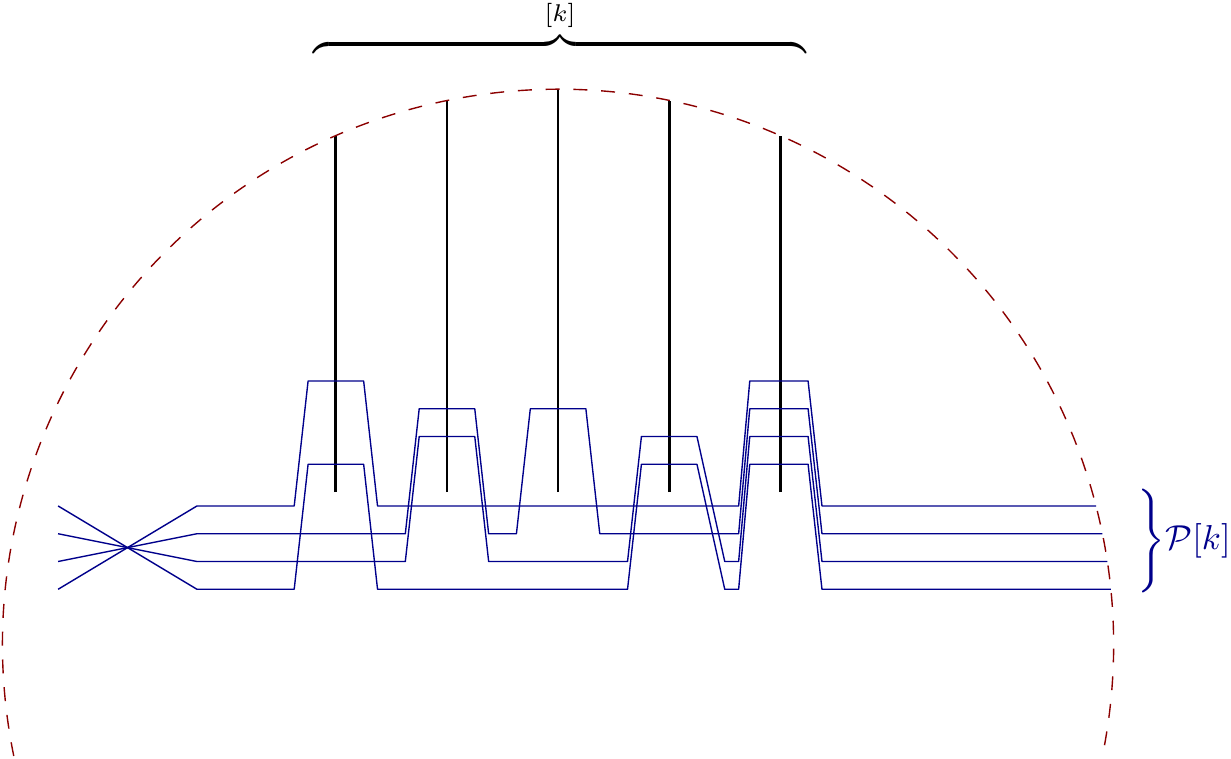}
	\caption{Sketch of the graph~$G_k$ from the proof~of Theorem~\ref{thm:outerstringNotSegment}.}
	\label{fig:outerstringGraph}
\end{figure}

Consider an intersection model with segments of the graph~$H_k\define G_k\left[[k]\right]$ induced by~the~vertices in~$[k]$, and let~$S_k$ be the set of segments representing~$[k]$ in the drawing. By Lemma~\ref{lemma:numberOfSeparableSubsetsOfSegs} there exist $\OO\left(k^4\right)$~different subsets of $S_k$ that are exactly transversed by any new, additional segment. Therefore any segment graph that contains~$H_k$ is going to have $\OO\left(k^4\right)$~different neighbourhoods on the vertices~$[k]$. Contrary to this, the graph~$G_k$ has $\OTheta\left(2^k\right)$~different neighbourhoods on the vertices~$[k]$. Since for any constant~$C$ there exists a~$k_0$ such that for any $k\geq k_0$ it holds that $Ck^4<2^k$, we obtain that $G_{k_0}$ is an outer-string graph but not a segment graph.
\end{proof}

Since there are outer-string graphs that are not segment graphs, they cannot be outer-segment graphs either. Note that the graph~$H$ discussed in Lemmas~\ref{lemma:triangle} and~\ref{lemma:ordGadgetUnitSegs} is a segment graph that is not an outer-string graph. This means that there is \textbf{no containment} between the class of segment graphs and the class of outer-string graphs.

%%%%%%%%%%%%%%%%%%%%%%%%%%%%%%%%%%%%%%%%%%%%%%%%%%%%%%%%%%%%%%%%%%%%%%%%%%%%%%%%%%%%%%%%%%%%%%%%%%%%%%%%%%%%%%%%%%%%%%%%%%%%%%%%%%%%%%%%%%%%%%%%%%%%%%%%%%%%%%%%%%%%%%%%%%%%
\section{Conclusions}
We have discussed strict containment between different classes of geometric intersection graphs. In a preliminary version of this work~\cite{CabelloJ16} we asked the following two questions:
\begin{itemize}
\item ``Let us mention one of the problems that we found more interesting in this context: is the class of outer-segment graphs a strict superclass of ray graphs?"
\item ``Consider the class $\mathcal{A}$ of intersection graphs of downward rays, that is, halflines that are contained in the halfplane $\{(x,y)\in \mathbb{R}^2 \mid y<c \}$ for some constant $c$. Let $\mathcal{B}$ be the class of intersection graphs of grounded segments, that is, segments contained in the halfplane $\{(x,y)\in \mathbb{R}^2 \mid y\ge0 \}$ with one endpoint on the $x$-axis. It is easy to see that $\mathcal{A}$ is a subset of $\mathcal{B}$. Is the containment proper?"
 \end{itemize}
These questions have been settled in the follow up work by Cardinal et al.~\cite{CardinalFMTV16}. 
An anonymous reviewer of this paper has raised an interesting, additional question: what can be said about segment graphs that can be realized with specific segment lengths? Or disks? The question for segments of fixed slopes has been considered by Cern{\'{y}} et al.~\cite{CernyKNP01}.
 
%%%%%%%%%%%%%%%%%%%%%%%%%%%%%%%%%%%%%%%%%%%%%%%
\section*{Acknowledgments}
We are grateful to Jean Cardinal, Jan Kratochv\'\i l and Bartosz Walczak for useful comments on the problems considered here. Special thanks go to Jan for explaining some parts of~\cite{KGK}.

%%%%%%%%%%%%%%%%%%%%%%%%%%%%%%%%%%%%%%%%%%%%%%%%%%%%%%%%%%%%%%%%%%%%%%%%%%%%%%%%%%%%%%%%

\end{document}